\documentclass[11pt]{amsart}

\usepackage[top=1.5in, bottom=1in, left=0.9in, right=0.9in]{geometry}

\usepackage{amscd,amsmath,amssymb,fancyhdr,color}
\usepackage[utf8]{inputenc}
\usepackage{amsfonts}
\usepackage{amsthm}
\usepackage{graphicx}
\usepackage{float}
\usepackage{subcaption}
\usepackage{verbatim}
\usepackage{indentfirst}
\usepackage{tikz}
\usepackage{tikz-cd}
\usetikzlibrary{matrix}
\usepackage[all]{xy}
\usepackage{enumerate}
\usepackage{csquotes}

\usepackage[backref=page]{hyperref}
\renewcommand*{\backref}[1]{}
\renewcommand*{\backrefalt}[4]{%
	\ifcase #1 (Not cited.)%
	\or        (Cited on page~#2.)%
	\else      (Cited on pages~#2.)%
	\fi}

\hypersetup{
	colorlinks   = true,
	citecolor    = magenta
}

\numberwithin{equation}{section}

\def\eqref#1{(\ref{#1})}

\newcommand{\C}{{\mathbb C}}
\newcommand{\R}{{\mathbb R}}

\def\1{\sqrt{-1}\:}

\newcommand{\cntrct}                % contraction with a vector field
{\hspace{2pt}\raisebox{1pt}{\text{$\lrcorner$}}\hspace{2pt}}

\renewcommand{\dim}{\operatorname{dim}}

\renewcommand{\Im}{\operatorname{Im}}

\renewcommand{\to}{\longrightarrow}

\newcounter{Mycounter}[section]
\newcounter{lemma}[section]
\newcounter{claim}[section]
\newcounter{sublemma}[section]
\newcounter{corollary}[section]
\newcounter{theorem}[section]
\newcounter{conjecture}[section]
\newcounter{proposition}[section]
\newcounter{definition}[section]
\newcounter{example}[section]
\newcounter{remark}[section]
\newcounter{problem}[section]
\newcounter{question}[section]
\makeatletter

\@addtoreset{equation}{section}

\@addtoreset{footnote}{section}

\makeatother

\usetikzlibrary{arrows,chains,matrix,positioning,scopes}

\makeatletter
\tikzset{join/.code=\tikzset{after node path={%
			\ifx\tikzchainprevious\pgfutil@empty\else(\tikzchainprevious)%
			edge[every join]#1(\tikzchaincurrent)\fi}}}
\makeatother

\tikzset{>=stealth',every on chain/.append style={join},
	every join/.style={->}}

\makeatletter
\newtheorem*{rep@theorem}{\rep@title}
\newcommand{\newreptheorem}[2]{%
	\newenvironment{rep#1}[1]{%
		\def\rep@title{\ref{##1}}%
		\begin{rep@theorem}}%
		{\end{rep@theorem}}}
\makeatother

\newreptheorem{theorem}{Theorem}

\begin{document}
	
	\newpage
	
	\title[Stein neighborhood bases for complex--totally real unions]{Stein neighborhood bases for complex--totally real unions}
	
	\author{Ovidiu Preda}
	\address{Ovidiu Preda \newline
		\textsc{\indent University of Bucharest, Faculty of Mathematics and Computer Science\newline 
			\indent 14 Academiei Str., Bucharest, Romania \newline
			\indent \indent and \newline
			\indent Institute of Mathematics ``Simion Stoilow'' of the Romanian Academy\newline 
			\indent 21 Calea Grivitei Street, 010702, Bucharest, Romania}}
	\email{ovidiu.preda@fmi.unibuc.ro; ovidiu.preda@imar.ro}
	
		\thanks{     \\[.1cm]
		}
	
	\date{\today}

	\begin{abstract}
		We study Stein neighborhood bases for unions of closed complex and totally real submanifolds, with particular attention to a linear configuration with noncompact intersection. We prove that the set $X=\{w=0\}\cup\{\Im z=\Im w=0\}\subset\C^2$ admits a Stein neighborhood basis. The proof starts from the explicit plurisubharmonic defining function $\rho(z,w)=(\Im w)^2+\sinh^2(\Im z)|w|^2$ and obtains the necessary non-uniform control at infinity by adding a locally uniformly convergent series of nonnegative plurisubharmonic terms constructed by Oka--Weil approximation. We also give an example of a closed complex submanifold and a closed totally real submanifold in $\C^2$ with compact intersection whose union has no Stein neighborhood basis, showing that the unrestricted union problem has a negative answer. 
	\end{abstract}
	
	\maketitle
	
	\hypersetup{linkcolor=blue}
	\tableofcontents

	\section{Introduction}
	
	Let $S$ be a closed subset of a complex manifold $Y$. A Stein neighborhood
basis of $S$ is a family of open Stein sets such that for every open set $U\subset Y$ containing
$S$ there is an open Stein set $\Omega$ in the family satisfying $S\subset\Omega\subset U$.
The existence of such a basis is very useful in several
complex variables, because it permits one to replace a geometrically singular set by
Stein domains on which the standard tools of holomorphic approximation and
sheaf cohomology are available.

The two types of submanifolds considered here are individually well
understood. Siu proved that every Stein subvariety admits a Stein neighborhood
\cite{Siu1976}. On the other hand, classical results on complexifications and
holomorphic approximation on totally real submanifolds give Stein
neighborhood bases for closed totally real submanifolds
\cite{Grauert1958,NirenbergWells1969,RangeSiu1974}. However, it does not necessarily follow 
that the union of such sets has the same property: analytic
continuation across the complex piece may interact with the geometry and
holomorphic hulls of the totally real piece.

An important precursor to our work is Boonstra's dissertation \cite[Chapter 3]{Boonstra1995}. In the
handlebody part of that work, Boonstra constructed strictly pseudoconvex
neighborhoods for standard configurations consisting of a strongly
pseudoconvex domain and an attached flat totally real handle. Forstneri\v{c}
and Kozak subsequently gave explicit strongly pseudoconvex neighborhoods of
quadratic strongly pseudoconvex domains with attached flat totally real
handles \cite{ForstnericKozak2003}. These results are closely related in
spirit to the present problem, but their complex piece is a full-dimensional
domain and the handle is attached along its boundary. 
In the model considered here,
the complex piece has positive codimension and its intersection with the
totally real piece is a noncompact real line. 
Slapar \cite{Slapar2004} constructed regular Stein neighborhood bases
for compact real surfaces smoothly embedded in complex surfaces whose
complex points are finite in number and all flat and hyperbolic.
He also applied his methods to certain unions of two totally real
planes in $\C^2$ intersecting only at the origin.
Starčič later constructed regular Stein
neighborhoods for the union of two totally real planes
$M=(A+iI)\R^2$ and $N=\R^2$ in $\C^2$, provided that the entries of the
real matrix $A$ are sufficiently small \cite{Starcic2016}. Although both
components in that setting are totally real, the use of strongly
pseudoconvex sublevel sets is close in spirit to the present construction. 

A second line of results concerns compact complex curves and holomorphic
convexity. Drinovec Drnov\v{s}ek and Forstneri\v{c} proved that a compact
complex curve with $\mathcal C^2$ boundary (and no boundaryless irreducible components) 
in a complex space admits a basis of open Stein neighborhoods \cite{DrinovecDrnovsekForstneric2007}.
Forstneri\v{c} later established a union theorem combining a compact bordered
complex curve $A$ (which has no boundaryless irreducible components) 
with a compact holomorphically convex set $K$, under the
decisive condition that $A\cap K$ be $\mathcal O(A)$-convex
\cite{Forstneric2022}. Relatedly, Forstneri\v{c} proved earlier a Stein-neighborhood
theorem for a closed Stein subvariety together with a compact holomorphically
convex set, again assuming holomorphic convexity of the intersection
\cite{Forstneric2005}. These theorems identify the absence of holes in the
intersection as a natural condition preventing a Hartogs phenomenon.

Col\c{t}oiu explicitly asked whether the union of the complex line
$\C\times\{0\}$ with the totally real plane $\R^2$ has a basis of
Stein neighborhoods in $\C^2$, and, more generally, whether the union
of a closed complex submanifold and a closed totally real submanifold
always admits a basis of Stein neighborhoods
\cite[Problem 5]{coltoiu_open_pb}. The noncompact intersection setting
requires particular care. For example, Kazama proved that
$\C^k\times\R^\ell$ has no Stein neighborhood basis in
$\C^k\times\C^\ell$ when $k,\ell\geq1$ \cite{Kazama1983}. Thus neither the
separate neighborhood theorems nor a naive product intuition settles the
union problem.

\vspace{10pt}

The purpose of this paper is twofold. First, we prove a positive result for
the basic noncompact linear configuration in $\C^2$. With coordinates
$(z,w)=(x+iy,u+iv)$, set
$$
A=\C\times\{0\},\qquad
M=\{(z,w)\in\C^2:y=0,\ v=0\},\qquad
X=A\cup M.
$$
Here $A$ is a complex line, $M$ is a maximal totally real plane, and
$A\cap M=\R\times\{0\}$ is noncompact. The proof begins with the explicit
nonnegative plurisubharmonic defining function
$\rho(z,w)=v^2+\sinh^2(y)|w|^2$, whose zero set is exactly $X$.
Since fixed sublevel sets of $\rho$ do not provide the required nonuniform
control at infinity, for an arbitrary neighborhood $U$ of $X$ we retain
$\rho$ as the initial term and add a locally uniformly convergent series of
functions $(\Im f_{n,j})^2$, where $f_{n,j}=wh_{n,j}$ and
$h_{n,j}$ is a polynomial with real coefficients. A localized Oka--Weil
argument makes each new batch arbitrarily small on the balls already treated
while making it large on the next compact part of $\C^2\setminus U$. This
yields a continuous plurisubharmonic function $\Psi_U$ such that
$$
\Psi_U|_X=0,\qquad
\Psi_U>1\quad\hbox{on }\C^2\setminus U.
$$
The connected component of $\{\Psi_U<1\}$ containing $X$ is then the
required Stein neighborhood.

Second, we show that the unrestricted statement is false. Namely, for
$A=\C\times\{0\}$ and $M=\{(z,w):|z|=1,\ w\in\R\}$, every Stein
neighborhood of $A\cup M$ contains $\overline\Delta\times\R$. Hence this
union has no Stein neighborhood basis. The obstruction is visible in the
intersection: $A\cap M=\partial\Delta\times\{0\}$ has a hole in the complex
line $A$. This gives a negative answer to Col\c{t}oiu's question without
additional hypotheses and places the positive linear model above in the same
holomorphic-convexity framework as the compact results of Forstneri\v{c} and
the joint work of Drinovec Drnov\v{s}ek and Forstneri\v{c}.

\vspace{10pt}

Section~\ref{sec:prelim} recalls the plurisubharmonic and
polynomial-convexity tools used in the proof. Section~\ref{sec:main}
establishes the positive model, and the following section gives the
Hartogs-type counterexample and discusses the role of the intersection.

	\section{Preliminaries}\label{sec:prelim}
	
	Throughout, $\mathcal O(Y)$ denotes the algebra of holomorphic functions on
a complex manifold $Y$, and all Euclidean balls are closed unless otherwise
stated.

\begin{definition}\label{def:Stein-basis}
Let $S$ be a closed subset of a complex manifold $Y$. We say that $S$ has a
Stein neighborhood basis in $Y$ if, for every open neighborhood $U$ of $S$ in
$Y$, there is an open Stein set $\Omega$ satisfying
$S\subset\Omega\subset U$. 
\end{definition}

\begin{definition}\label{def:totally-real}
A $\mathcal C^1$ real submanifold $M$ of an $n$-dimensional complex manifold
$Y$ is totally real if $T_pM\cap iT_pM=\{0\}$ for every $p\in M$. It is
maximal totally real if $\dim_{\R}M=n$.
\end{definition}

\subsection{Plurisubharmonic functions}

Let $u$ be a real-valued $\mathcal C^2$ function on an open set $V\subset \C^n$.
For $p\in V$ and $\xi\in\C^n$, its Levi form is
$$
\mathcal L_u(p;\xi)
=\sum_{j,k=1}^n
\frac{\partial^2u}{\partial z_j\partial\overline z_k}(p)
\xi_j\overline{\xi_k}.
$$
The function $u$ is plurisubharmonic if $\mathcal L_u(p;\xi)\geq0$ for every
$p$ and $\xi$, and it is strictly plurisubharmonic if the inequality is strict
for $\xi\neq0$. For continuous functions, plurisubharmonicity is defined by 
subharmonicity on complex lines. A continuous function $u$ is strictly plurisubharmonic 
if, locally near every point, $u-\delta|z|^2$ is plurisubharmonic for some $\delta>0$.

We shall repeatedly use the following elementary facts. If $f$ is
holomorphic, then $\Im f$ is pluriharmonic and
$\mathcal L_{(\Im f)^2}(p;\xi)=\frac12|df_p(\xi)|^2$; hence
$(\Im f)^2$ is plurisubharmonic. A locally uniformly convergent
sequence of continuous plurisubharmonic functions has a plurisubharmonic
limit. Also, if $u$ is plurisubharmonic and $\chi$ is convex and increasing,
then $\chi\circ u$ is plurisubharmonic whenever the composition is defined.

\begin{proposition}\label{prop:psh-sublevel}
Let $\psi$ be a continuous plurisubharmonic function on $\C^n$ and let
$c\in\R$. Every connected component of $\{\psi<c\}$ is Stein.
\end{proposition}

\begin{proof}
On $D=\{\psi<c\}$ the function
$\tau(\zeta)=|\zeta|^2+1/(c-\psi(\zeta))$ is a continuous strictly
plurisubharmonic exhaustion. Indeed, $t\mapsto1/(c-t)$ is convex and
increasing on $(-\infty,c)$; its composition with $\psi$ tends to infinity
at the finite boundary of $D$, while $|\zeta|^2$ tends to infinity at
infinity. The Levi problem, in its continuous-exhaustion form, shows that
every connected component of $D$ is Stein; see, for example,
\cite[Theorem~2.6.7, p.~46, and Theorem~4.2.8, p.~88]{Hormander1990}.
\end{proof}

\subsection{Holomorphic and polynomial convexity}

\begin{definition}\label{def:hulls}
If $K$ is a compact subset of a complex manifold $Y$, its
$\mathcal O(Y)$-hull is
$$
\widehat K_{\mathcal O(Y)}
=\{p\in Y:|f(p)|\leq\max_K|f|
   \text{ for every }f\in\mathcal O(Y)\}.
$$
The set $K$ is $\mathcal O(Y)$-convex if
$\widehat K_{\mathcal O(Y)}=K$. For $Y=\C^n$, this hull agrees with the
polynomial hull, and we say that $K$ is polynomially convex.
\end{definition}

A polynomial polyhedron in $\C^n$ is a compact set of the form
$P=\{z:|p_j(z)|\leq1,\ 1\leq j\leq N\}$ for finitely many polynomials
$p_1,\ldots,p_N$. We use the following standard consequences of
polynomial convexity and the Oka--Weil theorem
\cite[Lemma~2.7.4, p.~53, and Theorem~2.7.7, p.~55]{Hormander1990}.

\begin{proposition}\label{prop:Oka-Weil}
Let $K\subset\C^n$ be compact and polynomially convex. Every function holomorphic on a
neighborhood of $K$ can be approximated uniformly on $K$ by polynomials.
Moreover, every open neighborhood of $K$ contains a polynomial polyhedron
$P$ such that $K\subset\mathring{P}$. If $K$ and its neighborhood are
invariant under complex conjugation, then $P$ may be chosen invariant under
complex conjugation.
\end{proposition}

\vspace{10pt}

The additional conjugation-invariance assertion is an elementary consequence: after choosing a polynomial polyhedron $P$ with $K\subset\mathring{P}\subset P\subset U$, replace it by $P\cap\sigma(P)$, where $\sigma(z)=\overline z$. When $K$ is conjugation-invariant, this remains a polynomial polyhedron containing $K$ in its interior and contained in $U$.

\begin{lemma}\label{lem:add-points}
If $K\subset\C^n$ is compact, polynomially convex and
$F\subset\C^n\setminus K$ is finite, then $K\cup F$ is polynomially convex.
\end{lemma}

\begin{proof}
Fix $q\notin K\cup F$. Choose a polynomial $p$ with $p(q)=1$ and
$\max_K|p|<1$, and choose a polynomial $r$ with $r(q)=1$ and $r|_F=0$.
For sufficiently large $N$, the polynomial $rp^N$ has value $1$ at $q$ and
absolute value less than $1$ on $K\cup F$. Thus $q$ does not belong to the
polynomial hull of $K\cup F$.
\end{proof}

We now record the localization lemma which is an essential part of our main construction. In the
rest of the section, denote by $(z,w)$ the coordinates on $\C^2$, put $A=\{w=0\}$,
$M=\{\Im z=\Im w=0\}$, and
$X=A\cup M\subset\C^2$.

\begin{lemma}\label{lem:localized-polynomials}
Let $B_R$ be the closed Euclidean ball centered at the origin of radius $R$ in $\C^2$, let
$E$ be a compact subset of $\C^2\setminus(B_R\cup X)$, and let
$\epsilon>0$. There are finitely many polynomials
$f_1,\ldots,f_m$ of the form $f_j(z,w)=wh_j(z,w)$, with
$h_j\in\R[z,w]$, such that
$$
\sum_{j=1}^m(\Im f_j)^2>2
\quad\hbox{on }E,
\qquad
\sum_{j=1}^m|f_j|^2<\epsilon
\quad\hbox{on }B_R.
$$
\end{lemma}

\begin{proof}
Fix $p\in E$. Since $p\notin X$, we have $w(p)\neq0$ and
$p\neq\overline p$. By \ref{lem:add-points}, the compact set
$B_R\cup\{p,\overline p\}$ is polynomially convex. Choose pairwise
disjoint open sets $V_0,V_p,V_{\overline p}$ such that
$B_R\subset V_0$, $p\in V_p$, $\overline p\in V_{\overline p}$,
$\sigma(V_0)=V_0$, and $\sigma(V_p)=V_{\overline p}$, where
$\sigma(z,w)=(\overline z,\overline w)$. We also choose them so that
$V_p\cup V_{\overline p}$ does not meet $\{w=0\}$.

By \ref{prop:Oka-Weil}, there is a conjugation-invariant polynomial
polyhedron $P_p$ satisfying
$$
B_R\cup\{p,\overline p\}
\subset\mathring{P}_p
\subset P_p
\subset V_0\cup V_p\cup V_{\overline p}.
$$
On a neighborhood of $P_p$, prescribe the holomorphic function $H_p$ to be
$0$ on the part in $V_0$, $i\Lambda/w$ on the part in $V_p$, and
$-i\Lambda/w$ on the part in $V_{\overline p}$, where $\Lambda>2$ is fixed.
This function satisfies
$H_p(z,w)=\overline{H_p(\overline z,\overline w)}$.

Choose a neighborhood $N_p$ of $p$ whose closure is contained in
$V_p\cap\mathring{P}_p$. Carrying out this construction for every
$p\in E$, choose $p_1,\ldots,p_m$ such that
$N_{p_1},\ldots,N_{p_m}$ cover $E$.

By Oka--Weil, for each $j$ approximate $H_{p_j}$ uniformly on $P_{p_j}$ by
a polynomial $q_j$. Replacing $q_j$ with
$$
h_j(z,w)=\frac12\bigl(q_j(z,w)+
\overline{q_j(\overline z,\overline w)}\bigr),
$$
we obtain a polynomial with real coefficients and preserve the approximation.
Since $wH_{p_j}=i\Lambda$ near $\overline{N}_{p_j}$ and $H_{p_j}=0$ near
$B_R$, the approximations may be chosen close enough that, for
$f_j=wh_j$, we have $(\Im f_j)^2>2$ on $N_{p_j}$ and
$|f_j|^2<\epsilon/m$ on $B_R$. The two asserted inequalities follow.
\end{proof}

\subsection{Analytic discs}

We shall use the following form of the maximum principle. If
$g:\overline\Delta\to Y$ is continuous, holomorphic on $\Delta$, and $u$ is
plurisubharmonic on a neighborhood of $g(\overline\Delta)$, then
$u\circ g$ is subharmonic and
$\max_{\overline\Delta}u\circ g\leq\max_{\partial\Delta}u\circ g$.
Consequently, if $u$ is a plurisubharmonic exhaustion of a Stein domain and
the boundary circles of a family of analytic discs remain in a fixed compact
set, then the discs remain in a fixed compact sublevel set of $u$.

	\section{The main result}\label{sec:main}
	
	\begin{theorem}\label{thm_main}
		In $\C^2$ with coordinates $(z,w)=(x+iy,u+iv)$, we let $X=\{w=0 \}\cup \{y=0, v=0 \}$, which is the union of the complex plane $\C$ with a totally real plane $\R^2$. Then $X$ has a Stein neighborhood basis.
	\end{theorem}
	\begin{proof}
Consider $\rho(z,w)=v^2+\sinh^2(y)\,|w|^2$. Its zero set is exactly $X$. Indeed, if $\rho=0$, then $v=0$, and either $w=0$ or $y=0$.
Next, we prove that $\rho$ is plurisubharmonic. For simplicity, we write $h(y)=\sinh^2 y$.
The Levi matrix of $\rho$ is
\[ \mathcal L_\rho(z,w)=
\begin{bmatrix}
	\dfrac14 h''(y)|w|^2 & \dfrac{1}{2i}h'(y)w\\[2mm]
	-\dfrac{1}{2i}h'(y)\overline w & h(y)+\dfrac12
\end{bmatrix}.
\]
Since
\[ 
h'(y)=\sinh(2y),\quad
h''(y)=2\cosh(2y),\quad
h(y)+\frac12=\frac12\cosh(2y),
\]
the determinant is
\[ \det \mathcal L_\rho(z,w)=
\frac{|w|^2}{4}
\left[h''(y)\left(h(y)+\tfrac12\right)-h'(y)^2\right]
=
\frac{|w|^2}{4}
\left[\cosh^2(2y)-\sinh^2(2y)\right]
=
\frac{|w|^2}{4}\geq 0.
\]
Since the diagonal entries and the determinant of $\mathcal{L}_{\rho}$
are nonnegative, all its principal minors are nonnegative. Hence $\mathcal{L}_{\rho}$ 
is positive semidefinite, and therefore $\rho$ is plurisubharmonic.
The noncompactness of $X$ means that the uniform sublevel sets
$\{\rho<\varepsilon\}$ do not by themselves form a neighborhood basis.
We obtain the required nonuniform control by adding localized
plurisubharmonic terms to $\rho$.

Let $U$ be an arbitrary open neighborhood of $X$, put
$K=\C^2\setminus U$, and let $B_n$ denote the closed Euclidean ball centered 
at the origin of radius $n$. If $K=\varnothing$, take $\Omega_U=\C^2$ and the proof is
complete. Hence assume that $K\ne\varnothing$. If
$K\cap B_1\ne\varnothing$, choose $C>0$ such that
$C\rho>1$ on $K\cap B_1$. This is possible because $K\cap B_1$ is
compact and disjoint from $X=\{\rho=0\}$. If $K\cap B_1=\varnothing$,
take $C=1$. Set $\Psi_1=C\rho$.

Suppose inductively, for some $n\geq2$, that $\Psi_{n-1}$ is a
continuous nonnegative plurisubharmonic function which vanishes on $X$
and satisfies $\Psi_{n-1}>1$ on $K\cap B_{n-1}$. Set
$$
E_n=\{p\in K\cap B_n:\Psi_{n-1}(p)\leq1\}.
$$
The set $E_n$ is compact. Moreover, the induction hypothesis gives
$E_n\cap B_{n-1}=\varnothing$, while $K\cap X=\varnothing$ gives
$E_n\cap X=\varnothing$. Hence
$E_n\subset\C^2\setminus(B_{n-1}\cup X)$.

If $E_n=\varnothing$, set $m_n=0$ and $\Psi_n=\Psi_{n-1}$.
Otherwise, apply \ref{lem:localized-polynomials} with $R=n-1$ and
$\epsilon=2^{-n}$. We obtain polynomials
$f_{n,j}=wh_{n,j}$, where $h_{n,j}$ has real coefficients, such that
$$
\sum_{j=1}^{m_n}(\Im f_{n,j})^2>2
\quad\hbox{on }E_n,
\qquad
\sum_{j=1}^{m_n}|f_{n,j}|^2<2^{-n}
\quad\hbox{on }B_{n-1}.
$$
Define
$$
\Psi_n=\Psi_{n-1}
+\sum_{j=1}^{m_n}(\Im f_{n,j})^2.
$$
Each added summand is nonnegative and plurisubharmonic. It vanishes on
$A=\{w=0\}$ because $f_{n,j}$ is divisible by $w$, and it vanishes on
$M=\{y=0,v=0\}$ because $z,w\in\R$ there and $h_{n,j}$ has real
coefficients. Thus $\Psi_n$ vanishes on $X=A\cup M$. On $E_n$ the new
batch is greater than $2$, while on $(K\cap B_n)\setminus E_n$ we
already have $\Psi_{n-1}>1$. Consequently, $\Psi_n>1$ on
$K\cap B_n$. Furthermore,
$$
0\leq
\sum_{j=1}^{m_n}(\Im f_{n,j})^2
\leq
\sum_{j=1}^{m_n}|f_{n,j}|^2
<2^{-n}
\quad\hbox{on }B_{n-1}.
$$

For every fixed $N$, all batches with $n\geq N+1$ are uniformly bounded
by $2^{-n}$ on $B_N$. Therefore $(\Psi_n)$ converges locally uniformly
on $\C^2$ to the continuous plurisubharmonic function
$$
\Psi=C\rho+
\sum_{n=2}^{\infty}\sum_{j=1}^{m_n}
(\Im f_{n,j})^2,
$$
where an empty batch is interpreted as zero. Every summand vanishes on
$X$, so $\Psi|_X=0$. If $p\in K$, choose $n$ such that $p\in B_n$.
Then $\Psi(p)\geq\Psi_n(p)>1$. Hence $D=\{\Psi<1\}$ is an open
neighborhood of $X$ contained in $U$.

By \ref{prop:psh-sublevel}, every connected component of $D$ is Stein.
Since $A$ and $M$ are connected and $A\cap M\ne\varnothing$, the set
$X=A\cup M$ is connected. It is therefore contained in one connected
component $\Omega_U$ of $D$. Thus $X\subset\Omega_U\subset U$. Since
$U$ was arbitrary, $X$ has a Stein neighborhood basis.
\end{proof}

\section{A simple counterexample to the unrestricted statement}

Let $A=\mathbb C\times\{0\}\subset\mathbb C^2$ and $M=\{(z,w): |z|=1,\ w\in\mathbb R\}$. Then $A$ is a closed complex submanifold, and $M\cong S^1\times\mathbb R$ is a closed smooth totally real submanifold of $\mathbb C^2$. Set $X=A\cup M$.

For every \(t\in\mathbb R\), consider the analytic disc
\[
f_t:\overline{\Delta}\to\mathbb C^2,\qquad f_t(\zeta)=(\zeta,t).
\]
Its boundary satisfies $f_t(\partial\Delta)\subset M$, and for \(t=0\), $f_0(\overline{\Delta})\subset A$. Now let \(\Omega\) be any Stein neighborhood of \(X\). Let \(\psi\) be a plurisubharmonic exhaustion of \(\Omega\). Define $I=\{t\in\mathbb R: f_t(\overline{\Delta})\subset\Omega\}$. Then \(0\in I\). The set \(I\) is open by compactness. It is also closed: if \(t_j\to t_0\) and \(t_j\in I\), then the boundary circles
\[
f_t(\partial\Delta),\qquad |t-t_0|\le \epsilon,
\]
lie in a compact subset of \(M\subset\Omega\), so \(\psi\) is bounded above there. By the maximum principle applied to \(\psi\circ f_{t_j}\), the discs \(f_{t_j}(\overline{\Delta})\) stay in a fixed compact sublevel set of \(\psi\). Passing to the limit gives $f_{t_0}(\overline{\Delta})\subset\Omega$. Hence \(I=\mathbb R\).

Therefore every Stein neighborhood of $X$ contains $\overline{\Delta}\times\mathbb R$. In particular, every Stein neighborhood of $X$ contains $(0,1)$. But $(0,1)\notin X$. Taking $U=\mathbb C^2\setminus \overline{B((0,1),\varepsilon)}$
with $\varepsilon>0$ small gives an open neighborhood of $X$ containing no Stein neighborhood of $X$. Thus $X$ has no Stein neighborhood basis.

\vspace{15pt}

This is exactly the Hartogs-type obstruction: here $A\cap M=\{(z,0): |z|=1\}$ has a hole inside $A$. The compact analogue appears in Forstneri\v{c}'s extension of Siu's theorem: if $X_0$ is a closed Stein subvariety, $K$ is compact and holomorphically convex in a Stein neighborhood, and $K\cap X_0$ is $\mathcal O(X_0)$-convex, then $K\cup X_0$ has a Stein neighborhood basis. The same paper explains how failure of the intersection condition produces a bidisc-type Hartogs obstruction \cite{Forstneric2005}.

\section{Some remarks}

\begin{remark}\label{rem:localization-criterion}
	The additive argument generalizes in the following form. Let
$Y\subset\C^n$ be closed and connected, and suppose that there is a continuous
nonnegative plurisubharmonic function $\rho$ on $\C^n$ such that
$Y=\{\rho=0\}$. Assume in addition the following localization property:
for every $R>0$, every compact set
$E\subset\C^n\setminus(B_R\cup Y)$, and every $\epsilon>0$, there is a
continuous nonnegative plurisubharmonic function $\theta$ on $\C^n$ such
that
$$
\theta|_Y=0,\qquad \theta>2\quad\hbox{on }E,
\qquad \theta<\epsilon\quad\hbox{on }B_R.
$$
Then $Y$ has a Stein neighborhood basis.

Indeed, given an open neighborhood $U\supset Y$, one starts with a suitable
multiple $C\rho$ and repeats the defect-set construction $E_n$ from the proof
of \ref{thm_main}, applying the localization property at the $n$th step and
choosing the new term smaller than $2^{-n}$ on $B_{n-1}$. The resulting
series converges locally uniformly to a continuous
plurisubharmonic function $\Psi$ which vanishes on $Y$ and is greater than
$1$ on $\C^n\setminus U$. By \ref{prop:psh-sublevel}, the connected
components of $\{\Psi<1\}$ are Stein. The component containing the connected
set $Y$ is therefore a Stein neighborhood of $Y$ contained in $U$.
\end{remark}

\begin{remark}
	The model case of \ref{thm_main} can be generalized to all dimensions. Let
	$\mathbb C^n=\mathbb C^p_z\times\mathbb C^q_w$, where $p+q=n$.
	The subscripts $z$ and $w$ label the two coordinate blocks. Thus a point of
	$\mathbb C^n$ is written as $(z,w)$, where
	$z=(z_1,\ldots,z_p)\in\mathbb C^p$ and
	$w=(w_1,\ldots,w_q)\in\mathbb C^q$. Write $z=x+iy$ and
	$w=\sigma+it$, where $x,y\in\mathbb R^p$ and
	$\sigma,t\in\mathbb R^q$.
	All norms below are Euclidean. In particular,
	\[
	|\Im z|^2=|y|^2=\sum_{j=1}^p y_j^2,\qquad
	|\Im w|^2=|t|^2=\sum_{k=1}^q t_k^2,\qquad
	|w|^2=\sum_{k=1}^q|w_k|^2.
	\]
	Set $A_0=\mathbb C^p\times\{0\}$ and take the maximal totally real plane
	$M_0=\mathbb R^p\times\mathbb R^q$. Thus $A_0$ is given by $w=0$,
	where $0$ denotes the zero vector in $\mathbb C^q$, while $M_0$ is given by
	$\Im z=\Im w=0$. Let
	$h(z)=(e^{|\Im z|^2}-1)/2$ and define
	$\rho(z,w)=|\Im w|^2+h(z)|w|^2$.
	The function $\rho$ is nonnegative. It is straightforward to verify that $\{\rho=0\}=A_0\cup M_0$.
	
	We now write its Levi matrix explicitly. Regard $w$ and
	$\partial h=(h_{z_1},\ldots,h_{z_p})^T$ as column vectors, let $I_m$ denote
	the $m\times m$ identity matrix, and let the superscripts $T$ and $*$ denote
	transpose and Hermitian transpose, respectively. With the variables ordered
	as $(z,w)$, we have
	\[
	\mathcal L_\rho(z,w)=
	\begin{bmatrix}
		|w|^2\mathcal L_h(z) & (\partial h(z))w^T\\[2mm]
		\overline w\,(\partial h(z))^* &
		\left(h(z)+\dfrac12\right)I_q
	\end{bmatrix},
	\]
	where $\mathcal L_h=(h_{z_j\overline z_k})_{j,k=1}^p$. Indeed, the Levi
	matrix of $|\Im w|^2$ is $\frac12I_q$,
	$\rho_{z_j\overline z_k}=|w|^2h_{z_j\overline z_k}$, and
	$\rho_{z_j\overline w_\ell}=h_{z_j}w_\ell$. Put
	$g=h+\frac12=\frac12e^{|\Im z|^2}$.
	
	For clarity, recall that if
	$L=\left[\begin{smallmatrix}A&B\\ B^*&D\end{smallmatrix}\right]$ is
	Hermitian and $D>0$, then the Schur complement of $D$ in $L$ is
	$A-BD^{-1}B^*$, and $L\geq0$ if and only if
	$A-BD^{-1}B^*\geq0$. In the present case $D=gI_q>0$. Since
	$$g\mathcal L_h-\partial h(\partial h)^*
	=g^2\mathcal L_{|\Im z|^2}=\frac{g^2}{2}I_p,$$ its Schur
	complement is
	\[
	\begin{aligned}
		S
		&=|w|^2\mathcal L_h
		-\frac{|w|^2}{g}\,\partial h(\partial h)^*\\
		&=\frac{|w|^2}{g}
		\left(g\mathcal L_h-\partial h(\partial h)^*\right)
		=\frac12g|w|^2I_p\geq0.
	\end{aligned}
	\]
	It follows that $\mathcal L_\rho\geq0$, so $\rho$ is plurisubharmonic.
	
	This is precisely the block-matrix version of the calculation in the proof
	of \ref{thm_main}. Indeed, when $p=q=1$ and $h=h(y)$, one has
	$h_z=h'(y)/(2i)$ and $h_{z\overline z}=h''(y)/4$, and the preceding block
	matrix becomes
	\[
	\mathcal L_\rho(z,w)=
	\begin{bmatrix}
		\dfrac14h''(y)|w|^2 & \dfrac{1}{2i}h'(y)w\\[2mm]
		-\dfrac{1}{2i}h'(y)\overline w & h(y)+\dfrac12
	\end{bmatrix}.
	\]
	This is exactly the form of the Levi matrix in that proof, although the
	particular function $h$ is different. For a $2\times2$ Hermitian matrix
	$\left[\begin{smallmatrix}a&b\\ \overline b&g\end{smallmatrix}\right]$
	with $g>0$, the Schur-complement inequality
	$a-|b|^2/g\geq0$ is equivalent to $ag-|b|^2\geq0$, that is, to the
	nonnegativity of the determinant. Thus in the scalar case it is exactly the
	principal-minor criterion used in the proof of \ref{thm_main}.
	
	The localized approximation argument also extends to this model. Indeed, if
	$a\notin A_0\cup M_0$, then some coordinate $w_k(a)$ is nonzero and
	$a\ne\overline a$. In the proof of \ref{lem:localized-polynomials}, one
	replaces the factor $w$ by $w_k$ and performs the same conjugation-symmetric
	Oka--Weil approximation. The resulting holomorphic polynomials have the form
	$f=w_kq$, with $q$ having real coefficients. Hence $f$ vanishes on $A_0$
	and is real-valued on $M_0$, so the nonnegative plurisubharmonic correction
	$(\Im f)^2$ vanishes on $A_0\cup M_0$. Consequently, the
	localization property in \ref{rem:localization-criterion} holds, and
	$A_0\cup M_0$ has a Stein neighborhood basis.
	
	\vspace{10pt}
	
	A slightly more general model is this. For integers $0\leq s\leq p$ and
	$0\leq r\leq q$, split
	$z=(z',z'')\in\mathbb C^s\times\mathbb C^{p-s}$ and
	$w=(u,v)\in\mathbb C^r\times\mathbb C^{q-r}$. Here $z'$, $z''$, $u$, and
	$v$ are vectors in the indicated complex Euclidean spaces. In particular,
	$u$ and $v$ are complex coordinate blocks of $w$; they are not its real
	and imaginary parts. More explicitly,
	\[
	\begin{aligned}
		|\Im z'|^2
		&=\sum_{j=1}^s(\Im z'_j)^2,
		&\qquad |z''|^2
		&=\sum_{j=1}^{p-s}|z''_j|^2,\\
		|\Im u|^2
		&=\sum_{j=1}^r(\Im u_j)^2,
		& |v|^2
		&=\sum_{k=1}^{q-r}|v_k|^2.
	\end{aligned}
	\]
	Define
	\[
	A_0=\{w=0\}=\{u=0,\ v=0\},\qquad
	M_0=\{z'\in\mathbb R^s,\ z''=0,\ u\in\mathbb R^r,\ v=0\}.
	\]
	Then $M_0$ is totally real and $A_0\cap M_0$ is naturally identified with
	$\mathbb R^s$. Put
	$\tau(z)=|\Im z'|^2+|z''|^2$,
	$h(z)=(e^{\tau(z)}-1)/2$, and
	$\rho(z,u,v)=|\Im u|^2+|v|^2
	+h(z)(|u|^2+|v|^2)$.
	Again $\rho\geq0$. As in the previous case, a straightforward computation shows that
	$\{\rho=0\}=A_0\cup M_0$.
	
	We also give the Levi calculation in this case. Set
	$R=|u|^2+|v|^2$ and $g=h+\frac12=\frac12e^\tau$. With the variables
	ordered as $(z,u,v)$, the Levi matrix is
	\[
	\mathcal L_\rho(z,u,v)=
	\begin{bmatrix}
		R\mathcal L_h
		&(\partial h)u^T&(\partial h)v^T\\[2mm]
		\overline u\,(\partial h)^*
		&gI_r&0\\[2mm]
		\overline v\,(\partial h)^*
		&0&(h+1)I_{q-r}
	\end{bmatrix}.
	\]
	Indeed, $|\Im u|^2$ contributes $\frac12I_r$ to the
	$u\overline u$ block, whereas $|v|^2$ contributes $I_{q-r}$ to the
	$v\overline v$ block. Thus the lower-right block is
	$D=\operatorname{diag}(gI_r,(h+1)I_{q-r})\geq gI_q>0$. As above,
	$g\mathcal L_h-\partial h(\partial h)^*=g^2\mathcal L_\tau$, where
	$\mathcal L_\tau=\operatorname{diag}(\frac12I_s,I_{p-s})>0$. Since
	$h+1\geq g$, the Schur complement of $D$ satisfies
	\[
	\begin{aligned}
		S
		&=R\mathcal L_h-
		\left(\frac{|u|^2}{g}+\frac{|v|^2}{h+1}\right)
		\partial h(\partial h)^*\\
		&\geq R\mathcal L_h-\frac{R}{g}\,
		\partial h(\partial h)^*\\
		&=\frac{R}{g}
		\left(g\mathcal L_h-\partial h(\partial h)^*\right)
		=Rg\mathcal L_\tau\geq0.
	\end{aligned}
	\]
	Therefore $\mathcal L_\rho\geq0$, and $\rho$ is plurisubharmonic. This is the same
	block Schur-complement argument as for the first model.
	
	For completeness, the localized approximation can be obtained as follows.
	At a point $a\notin A_0\cup M_0$, if $v_k(a)\ne0$ for some $k$, use a
	correction of the form $f=v_kq$; its factor $v_k$ vanishes on both $A_0$
	and $M_0$. If $v(a)=0$ but $z''_k(a)\ne0$, the fact that
	$a\notin A_0$ implies that $u(a)\ne0$. Choose $j$ with $u_j(a)\ne0$ and
	use $f=u_jz''_kq$. This function vanishes on $A_0$ because $u=0$ there,
	and on $M_0$ because $z''=0$ there. In these two cases, we use ordinary 
	Oka--Weil approximation without conjugation symmetrization. For the chosen factor 
	$P=v_k$ or $P=u_jz''_k$, we prescribe $0$ near $B_R$ and $i\Lambda/P$ near $a$, 
	and repeat the polynomial-polyhedron and finite-cover construction from 
	\ref{lem:localized-polynomials}.
	
	In the remaining case
	$v(a)=z''(a)=0$, choose $j$ with $u_j(a)\ne0$. Since
	$a\notin A_0$, such a $j$ exists; and since $a\notin M_0$, either $z'$ or
	$u$ has a nonreal coordinate, so $a\ne\overline a$. The
	conjugation-symmetric construction applied to $u_jq$, with $q$ having real
	coefficients, gives a holomorphic function $f=u_jq$ which vanishes on $A_0$
	and is real-valued on $M_0$. In all three cases the resulting
	plurisubharmonic correction $(\Im f)^2$ vanishes on
	$A_0\cup M_0$. The proof of \ref{lem:localized-polynomials} now applies in
	each case, so \ref{rem:localization-criterion} shows that this linear model
	also has a Stein neighborhood basis.
\end{remark}

	\vspace{20pt}

\textbf{Disclosure of AI assistance.} The author acknowledges the use of general-purpose large language models as an aid in documentation, bibliographic searches, testing the suitability of various candidate functions during the development of the proof, and proof-checking.

\end{document}